\documentclass[11pt,final]{article}
\usepackage[utf8]{inputenc}
\usepackage[english]{babel}
\usepackage[T1]{fontenc}
\usepackage{lmodern}

\usepackage{amsmath}
\usepackage{amssymb}
\usepackage{latexsym}
\usepackage{amsthm}
\usepackage{fullpage}
\usepackage{graphicx}
\usepackage[colorinlistoftodos,textsize=tiny,obeyFinal]{todonotes}
\usepackage{marginnote}
\usepackage[unicode,breaklinks]{hyperref}   

\numberwithin{equation}{section}

\newtheorem{example}{Example}[section]
\newtheorem{theorem}{Theorem}[section]
\newtheorem{lemma}{Lemma}[section]

\newtheorem{proposition}{Proposition}[section]

\newtheorem{remark}[example]{Remark}
\newtheorem{definition}[example]{Definition}
\newtheorem{corollary}[example]{Corollary}

\newcommand{\R}{{\mathbb R}}
\newcommand{\N}{{\mathbb N}}

\newcommand{\be}{\begin{eqnarray}}

\newcommand{\ee}{\end{eqnarray}}

\renewcommand{\d}{{\rm d}}

\newcommand{\md}{{\rm d}}

\renewcommand{\O}{\Omega}

\newcommand{\wto}{\rightharpoonup}

\renewcommand{\wto}{\rightharpoonup}

\newcommand{\DD}{{\mathcal D}}
\newcommand{\EE}{{\mathcal E}}
\newcommand{\II}{{\mathcal I}}
\newcommand{\NN}{{\mathcal N}}
\newcommand{\QQ}{{\mathcal Q}}
\newcommand{\VV}{{\mathcal{A}}}
\newcommand{\ZZ}{{\mathcal Z}}
\newcommand{\YY}{{\mathcal{A}}}
\newcommand{\p}{\partial}
\newcommand{\qqquad}[0]{\qquad\qquad}

\newcommand{\NT}[0]{\notag}

\newcommand{\dist}{{\operatorname{dist}}}

\DeclareMathOperator{\cof}{Cof}

\author{Martin Kru\v{z}\'{\i}k\footnote{The Czech Academy of Sciences, Institute of Information Theory and Automation,
, Pod vod\'{a}renskou
v\v{e}\v{z}\'{\i}~4, CZ-182~08~Praha~8, Czechia (corresponding
address) \& Faculty of Civil Engineering, Czech Technical
University, Th\'{a}kurova 7, CZ-166~ 29~Praha~6, Czechia}\ \
Petr Pelech\footnote{Mathematical Institute, Charles University, Sokolovsk\'{a} 83, 186~00 Praha 8, Czechia}\ \
Anja Schl\"{o}merkemper\footnote{Institute of Mathematics,
University of W\"{u}rzburg, Emil-Fischer-Stra\ss e 40, 97074 W\"{u}rzburg,    Germany}
}
\title{Shape memory alloys as gradient-polyconvex materials}
\begin{document}

\maketitle
\begin{abstract}
 We show existence of an energetic solution to a model of shape memory alloys
 in~which the~elastic energy is described by means of a gradient-polyconvex functional.
 This allows us to show existence of a solution based on weak continuity of nonlinear minors
 of deformation gradients in Sobolev spaces.
 Resulting deformations are orientation-preserving and injective everywhere in a domain representing the specimen.
\end{abstract}

\section{Introduction}
Hyperelasticity is a special area of Cauchy elasticity where one assumes
that the first Piola-Kirchhoff stress tensor $S$ possesses a potential
(called stored energy density) $W : \R^{3\times 3} \to [-w, \infty]$,
for some $w\geq 0$.
In other words, 
\begin{align} \label{1stPK}
  S:=\frac{\partial W(F)}{\partial F}
\end{align}
on its domain, where $F\in\R^{3\times 3}$ is such that $\det F>0$.
This concept emphasizes that all work done by external loads on the specimen
is stored in it. 
The principle of frame-indifference requires
that $W$ satisfies for all $F\in\R^{3\times 3}$
and all proper rotations $R\in{\rm SO}(3)$ 
\begin{align*}
  W(F)=W(RF)=\tilde W(F^\top F)=\tilde W(C)\ ,
\end{align*}
where $C:=F^\top F$ is the right Cauchy-Green strain tensor
and $\tilde W: \R^{3\times 3} \to [-w,\infty]$.
Additionally, every elastic material is assumed to resist extreme compression,
which is modeled by assuming
\begin{align} \label{det} 
W(F)\to+\infty \text{ if }\det F\to 0_+ .
\end{align}
Let the reference configuration be a bounded Lipshitz domain
$\Omega\subset \R^3$.
If we consider a deformation $y:\bar\O \to \R^3$,
which is a mapping that assigns to each point in the closure
of the reference configuration $\bar \Omega$ its position after deformation,  
solutions to corresponding elasticity equations can be formally found
by minimizing an energy functional
\begin{align} \label{energy-funct}
I(y):=\int_\O W(\nabla y(x))\,\md x-\ell(y)
\end{align}
over a class of admissible deformations.
Here $\ell$ is a functional on the set of deformations expressing
(in a simplified way) the work of external loads on the specimen
and $\nabla y$ is the deformation gradient which quantifies the strain.
We only allow for deformations which are orientation-preserving,
i.e. if $a,b,c\in\R^3$ satisfy  $(a\times b)\cdot c>0$,
then $(Fa\times Fb)\cdot Fc>0$ for every $F:=\nabla y(x)$ and $x\in\O$.
Which means that $\det F>0$.
This condition can be expressed by extending $W$ by infinity to matrices
with nonpositive determinants, 
\begin{align}
  \label{det2} 
  W(F):=+\infty \text{ if }\det F\le 0 .
\end{align}

In view of \eqref{1stPK}, \eqref{det}, and \eqref{det2}
we see that $W:\R^{3\times 3}\to [-w,+\infty]$, for some $w\ge 0$,
is continuous in the sense that if $F_k\to F$ in $\R^{3\times 3}$
for $k\to+\infty$,
then $\lim_{k\to+\infty}W(F_k)= W(F)$.
Furthermore, $W$ is differentiable on the set of matrices
with positive determinants.  

A key question immediately appears:
Under which conditions does the functional $I$ in \eqref{energy-funct}
possess minimizers?
Relying on the direct method of the calculus of variations,
the usual approach to address this question is to study
(weak) lower semicontinuity of the functional $I$
on appropriate Banach spaces containing the admissible deformations.
For definiteness,
we assume that $y\mapsto -\ell(y)$ is weakly sequentially lower semicontinuous.
Thus the question reduces do a discussion of~the~assumptions on $W$.
It is well known that \eqref{det} prevents us from assuming convexity of $W$.
See e.g.~\cite{dacorogna} or the recent review \cite{benesova-kruzik}
for a detailed exposition of weak lower semicontinuity.
In his seminal contribution \cite{ball77},
J.M.~Ball defined a polyconvex stored energy density $W$
by assuming that there is a convex and lower semicontinuous function
$\bar W:\R^{19}\to [-w,+\infty]$ such that 
\begin{align*}
  W(F):=\bar W(F,\cof F,\det F)\ .
\end{align*}
Here $\cof F$ is the~cofactor matrix of $F$, which for $F$ being invertible satisfies Cramer's rule   
\begin{align*}
\cof F   =     (\det F) F^{-\top}\ .
\end{align*}
It is well-known that polyconvexity is satisfied for a  large class of constitutive functions and allows for existence of minimizers of $I$ under \eqref{det} and \eqref{det2}. On the other hand, there are still situations where  polyconvexity cannot be adopted. A prominent example are shape-memory alloys, see e.g.\,\cite{ball-james,bhattacharya, mueller}, where $W$ has the so-called multi-well structure. Namely, there is a high-temperature phase called austenite, which is usually of cubic symmetry, and a low-temperature phase called martensite, which is less symmetric and exists in more variants, e.g., in three for the tetragonal structure (NiMnGa) or in twelve for the monoclinic one (NiTi).
We can assume that 
\begin{align}\label{stored-sma}
W(F):=\min_{0\le i\le M}W_i(F)\ ,
\end{align}
where $W_i:\R^{3\times 3}\to[-w_i,+\infty]$, $w_i\ge 0$, is the stored energy density of the $i$-th variant of martensite if $i>0$, and $W_0$ is the stored energy density of the austenite.  For every admissible $i$, $W_i(F)=-w_i$ if and only if $F=RF_i$ for a given matrix $F_i\in\R^{3\times 3}$ and an arbitrary proper rotation $R\in{\rm SO}(3)$. 

Let us emphasize that \eqref{stored-sma} ruins even generalized notions
of convexity as e.g. rank-one convexity
(we recall that rank-one convex functions are convex on line segments
whose endpoints differ by~a~rank-one matrix
and that rank-one convexity is a necessary condition for polyconvexity;
cf.\,\cite{dacorogna}, for instance).
Namely, it is observed (see e.g.\,\cite{ball-james,bhattacharya})
that $w_i=w_j$
whenever $i,j\ne 0$
and that there is a~proper rotation $R_{ij}$ such that 
${\rm rank}(R_{ij}F_i-F_j)=1$.
Hence, generically, $W(R_{ij}F_i)=W(F_j)=-w_i$,
but $W(F)>-w_i$ if $F$ is on the line segment between $R_{ij}F_i$ and $F_j$;
however, not having a convexity property at hand
that implied existence of minimizers
is in~accordance with experimental observations for these alloys.

Indeed, nonexistence of a minimizer corresponds to the formation
of microstructure of strain-states
which is mathematically manifested via faster and faster oscillation
of deformation gradients in minimizing sequences
driving the functional $I$ to its infimum.
One can then formulate a minimization problem
for a lower semicontinuous envelope of $I$,
the so-called relaxation, see, e.g., \cite{dacorogna}.
Such a relaxation yields information of the effective behaviour
of the material and on the set of~possible microstructures.
Thus relaxation is not only an important tool for mathematical analysis,
but also for applications.
For numerical considerations it is a challenging problem,
because the relaxation formula is generically not obtained in a closed form.
Further difficulties come from the fact
that a sound mathematical relaxation theory is developed only
if $W$ has $p$-growth; that is,
for some $c>1$, $p \in (1,+\infty)$ and all $F\in \R^{3 \times 3}$
the inequality
\begin{align*}
\frac1c(|F|^p-1) \leq W(F) \leq c(1+|F|^p)
\end{align*}
is satisfied, which in particular implies that $W<+\infty$.
We refer, however, to \cite{benesova-kruzik,conti-dolzmann, krw}
for results allowing for infinite energies.
Nevertheless, these works include other assumptions which severely restrict their usage.
Let us point out that the right Cauchy-Green strain tensor $F^\top F$ maps SO$(3)F$
as well as (O$(3)\setminus$SO$(3)$)$F$ to the same point.
Here O$(3)$ are orthogonal matrices with  determinant $\pm 1$.
Thus, for example, 
$F\mapsto |F^\top F-\mathbb{I}|$ is minimized on two energy wells,
on SO$(3)$ and also on O$(3)\setminus$SO$(3)$.
However, the latter set is not acceptable in elasticity because the corresponding minimizing affine deformation is a mirror reflection.
In~order to~distinguish between these two~wells,
it is necessary to~incorporate $\det F$ in~the~model properly.

Besides relaxation, another approach guaranteeing existence of minimizers is to resort to nonsimple materials, i.e., materials whose stored energy density depends (in a convex way)  on higher deformation gradients. This idea goes back to Toupin \cite{Toupin:62,Toupin:64} and is used in many works from then on \cite{BCO,vidoli,forest,Silh88PTNB}, including work on shape-memory alloys \cite{ball-crooks,ball-mora}.
   Simple examples are functionals of the form    
\begin{align*}
I(y):=\int_\O W(\nabla y(x))+\varepsilon|\nabla^2 y(x)|^p\,\md x-\ell(y)\ ,
\end{align*} 
where $\varepsilon>0$.  Obviously, the second-gradient term brings additional compactness to the problem, which allows  to require only strong lower semicontinuity 
of the term 
$$\nabla y\mapsto \int_\O W(\nabla y(x))\,\md x$$
   for existence of minimizers.    

Here we follow a different approach recently suggested in \cite{bbmkas},
which is a natural extension of~polyconvexity exploiting weak continuity
of minors in Sobolev spaces.
Instead of the~full second gradient,
it is assumed that the stored energy of the material depends on the deformation gradient $\nabla y$
and on gradients of nonlinear minors of $\nabla y$, i.e., on $\nabla[\cof\nabla y]$ and on $\nabla[\det\nabla y]$.
The corresponding functionals are then called gradient polyconvex.
While we assume convexity of the stored energy density in the two latter terms,
this is not assumed in the $\nabla y$ variable.
The advantage is that minimizers are elements of Sobolev spaces $W^{1,p}(\Omega, \R^3)$
and no higher regularity is required.

The following example is inspired from \cite{bbmkas}.
It shows that there are maps with smooth nonlinear minors
whose deformation gradient is \emph{not} a  Sobolev map.
Hence, gradient polyconvex energies are more general
than second-gradient ones.  

\begin{example} \label{exa-grad-poly} 
  Let $\Omega = (0,1)^3$.
  For~functions $f,g : (0,1) \to (0, +\infty)$ to~be~specified later,
  let us consider the deformation
\begin{align*}
  y(x_1,x_2,x_3) := \left(x_1, x_2 f(x_1), x_3 g(x_1) \right).
\end{align*}
Then
\begin{align*}
  \nabla y(x_1,x_2,x_3)
  &=
  \left(
    \begin{array}{ccc}
      1 & 0 & 0 \\ 
      x_2 f'(x_1) & f(x_1) & 0 \\ 
      x_3 g'(x_1) & 0 & g(x_1)
    \end{array}
  \right), \\
  \cof \nabla y(x_1,x_2,x_3)
  &=
  \left(
    \begin{array}{ccc}
      f(x_1) g(x_1) &  - x_2 f'(x_1) g(x_1) & - x_3 f(x_1) g'(x_1)  \\ 
      0 & g(x_1) & 0 \\ 
      0 & 0 & f(x_1)
    \end{array}
  \right)
\end{align*}
and
\begin{align*}
  \det \nabla y(x_1,x_2,x_3)
  = f(x_1) g(x_1) > 0\ .
\end{align*}
Finally, the~non-zero entries of~$\nabla^2 y(x_1, x_2, x_3)$ are
\begin{align} \label{sec-non-zero}
    x_2 f''(x_1),\quad f'(x_1), \quad x_3 g''(x_1),\quad g'(x_1) \ .
\end{align}
Note that we have in particular
\begin{align*}
  |\nabla^2 y(x_1, x_2, x_3)| \geq |x_2| |f''(x_1)|.
\end{align*}
Any functions $f,g$ such that $y \in W^{1,p}(\O; \R^3)$,
\todo{Please check the exponents of $q,r,s$. AS Seems to be OK. PP}
$\cof \nabla y \in W^{1,q}(\O; \R^{3\times 3})$,
$\det \nabla y \in W^{1,r}(\O; (0,\infty))$,
$(\det \nabla y)^{-s} \in L^1(\O)$ for some $p,q,r\geq 1$ and $s>0$, but such that one of the quantities in \eqref{sec-non-zero} is not a function in $L^p(\O)$ yield a useful example since then $y\notin W^{2,p}(\O; \R^3)$. To be specific, we choose for $1>\varepsilon > 0$
\begin{align*}
  f(x_1) = x_1^{1-\varepsilon}
  \quad
  \text{and}
  \quad
  g(x_1) = x_1^{1 + \varepsilon}.
\end{align*}
Hence
\begin{align*}
  f'(x_1) &= (1-\varepsilon) x_1^{-\varepsilon},
  &
  g'(x_1) &= (1+\varepsilon) x_1^{\varepsilon}, \\
  f''(x_1) &= -\varepsilon (1-\varepsilon) x_1^{-1-\varepsilon}
  &
  g''(x_1) &= \varepsilon (1+\varepsilon) x_1^{-1+\varepsilon}.
\end{align*}
Since $x_2 f''(x_1)$ is not integrable, we have $\nabla^2 y\not\in L^1(\O;\R^{3\times 3\times 3})$ and thus $y\not\in W^{2,1}(\O;\R^3)$.
We have only $y\in W^{1,p}(\O;\R^3)\cap L^\infty(\O;\R^3)$
for every $1 \le p < 1/\varepsilon$.
Moreover, direct computation shows that both
$\cof \nabla y$ and $\det \nabla y$ lie in~$W^{1,\infty}$.
Finally, $\det \nabla y = x_1^2 >0$
and $(\det\nabla y)^{-s} \in L^1(\O)$ for~all $0 < s < 1/2$.

Therefore, for any $r,q \geq 1$, $s>0$, requiring a~deformation $y:\Omega \to \R^3$ to~satisfy $\det \nabla y \in W^{1,r}(\Omega)$, $(\det \nabla y)^{-s} \in L^1(\O)$ and $\cof\nabla y \in W^{1,q}(\Omega; \R^{3 \times 3})$ is a weaker assumption than $y \in W^{2,1}(\Omega; \R^{3})$.   
\end{example}

\bigskip

 \section{Gradient polyconvexity}
We start with a definition of gradient polyconvexity.

\begin{definition}[See \cite{bbmkas}] \label{def-gpc}
Let  Let $\O\subset\R^3$ be a bounded open domain.   Let $\hat{W}:\R^{3\times 3}\times\R^{3\times 3\times 3}\times\R^3\to\R\cup\{+\infty\}$ be a lower semicontinuous function. The functional
\begin{align}\label{full-I}
J(y) = \int_\Omega \hat W(\nabla y(x), \nabla[ \cof \nabla y(x)], \nabla[ \det \nabla y(x)]) \d x,
\end{align}
defined for any measurable function $y: \Omega \to \R^3$ for which the weak derivatives $\nabla y$, $\nabla[ \cof \nabla y]$, $\nabla[ \det \nabla y]$ exist and are integrable is called {\em gradient polyconvex} if the function $\hat{W}(F,\cdot,\cdot)$ is convex for every $F\in\R^{3\times 3}$.
\end{definition}

With $J$ defined as in    \eqref{full-I} and a functional  $y\mapsto \ell(y)$ expressing the (negative) work of external loads we set 
\begin{align}\label{functional-sma}
I(y):=J(y)-\ell(y)\ .
\end{align}

Besides convexity properties, the~results of~weak lower semicontinuity
of $I$ on~$W^{1,p}(\O;\R^3)$ (for~$1\le p<+\infty$) rely on suitable coercivity properties. Here we assume that 
there are numbers $p,q,r>1$ and $c,s>0$ such that 
\begin{align}\label{growth-graddet2}
\hat W(F,\Delta_1,\Delta_2)\ge\begin{cases}
c\big(|F|^p  +|\cof F|^q+(\det F)^r+ (\det F)^{-s}+|\Delta_1|^q + |\Delta_2|^r\big)
&\text{ if }\det F>0,\\
+\infty&\text{ otherwise.} \end{cases}
\end{align}

   The following existence result is taken from \cite{bbmkas}. For the reader's convenience, we provide a proof below.  
\begin{proposition}
\label{prop-grad-poly}
Let $\O\subset\R^3$ be a bounded Lipschitz domain,  
and let $\Gamma=\Gamma_0\cup\Gamma_1$ be a $\md A$-measurable partition
of $\Gamma=\partial\O$ with the area of $\Gamma_0>0$.
Let further $-\ell:W^{1,p}(\O;\R^3)\to\R$
be a~weakly lower semicontinuous functional satisfying
for some $\tilde C>0$ and $1\le \bar p<p$
\begin{align*}
  \forall y\in W^{1,p}(\O;\R^3)
  :
  \quad
  \ell(y)\le \tilde C\|y\|^{\bar p}_{W^{1,p}(\O;\R^3)}\ .
\end{align*}
Further let $J$, as in \eqref{full-I}, be gradient polyconvex on $\Omega$
and such that there is a $\hat W$ as in Definition~\ref{def-gpc} which in addition satisfies \eqref{growth-graddet2} for    $p> 2$, $q\ge\frac{p}{p-1}$, $r>1$,  $s>0$.    Moreover,    assume that for some given measurable function $y_0: \Gamma_0 \to \R^3$ the following set 
 \begin{align*}
\mathcal{A}:&=\{y\in W^{1,p}(\O;\R^3):\ \cof \nabla y\in W^{1,q}(\O;\R^{3\times 3}),\ \det \nabla y\in W^{1,r}(\O),\nonumber\\
& \qquad (\det\nabla y)^{-s}\in L^1(\O),\  \det\nabla y>0\mbox{ a.e. in $\Omega$},\ y=y_0\mbox{ on }\Gamma_0\}
\end{align*}
is nonempty. If $\inf_{\mathcal{A}} I<\infty$  for $I$ from \eqref{functional-sma}, then
 the functional  $I$  has a minimizer on $\mathcal{A}$.
\end{proposition}

\begin{proof}
Our proof closely follows the approach in \cite{bbmkas}.
Let $\{y_k\}\subset\mathcal{A}$ be a minimizing sequence of~$I$. 
Due to coercivity assumption \eqref{growth-graddet2} and the Dirichlet boundary conditions on $\Gamma_0$, we obtain that  
\begin{align}\label{bound} &\sup_{k\in\N}\big(\|y_k\|_{W^{1,p}(\O;\R^3)} +\|\cof\nabla y_k\|_{W^{1,q}(\O;\R^{3\times 3})}\nonumber\\
&+\|\det\nabla y_k\|_{W^{1,r}(\O)}+\|(\det\nabla y_k)^{-s}\|_{L^1(\O)}\big)<\infty\ . \end{align}
Hence, by standard results on weak convergence of minors,
see e.g. \cite[Thm.~7.6-1]{ciarlet},
there are (non-relabeled) subsequences such that
\begin{align*}
  y_k \wto y \text{ in } W^{1,p}(\O;\R^3),
  \quad
  \cof\nabla y_k \wto \cof\nabla y \text{ in } L^q(\O;\R^{3\times 3}),
  \quad
  \det\nabla y_k \wto \det\nabla y \text{ in } L^r(\O)
\end{align*}
for $k\to\infty$.
Moreover, since bounded sets in Sobolev spaces
are weakly sequentially compact,
\begin{align} \label{grad-det-cof-weak}
  \cof\nabla y_k \wto H \text{ in } W^{1,q}(\O;\R^{3\times 3}),
  \quad
  \det\nabla y_k \wto D \text{ in } W^{1,r}(\O)
\end{align}
for some $H\in W^{1,q}(\O;\R^{3\times 3})$ and
$D \in W^{1,r}(\O)$.
Since the~weak limit is unique, we have $H=\cof \nabla y$
and $D = \det \nabla y$.
By compact embedding also
$\cof\nabla y_k\to H$ in $L^q(\O;\R^{3\times 3})$
and hence we obtain a (non-relabeled) subsequence such that for~$k\to \infty$
\begin{align} \label{cof-conv}
  \cof \nabla y_k \to \cof \nabla y
  \quad
  \text{a.e. in } \O.
\end{align} 
Since by Cramer's formula $\det (\cof \nabla y) = (\det\nabla y)^2$,
we have for~$k\to\infty$
\begin{align} \label{det-conv}
  \det\nabla y_k \to \det\nabla y
  \quad
  \text{a.e. in } \Omega.
\end{align}

Next we show that $y$ belongs to the set of admissible functions $\mathcal{A}$.
Notice that $\det\nabla y\ge 0$ since $\det \nabla y_k >0$ for any $k\in \N$.
Further, the conditions \eqref{growth-graddet2}, \eqref{bound},
and the Fatou lemma imply that 
\begin{align*}
  +\infty
  >
  \liminf_{k\to\infty} J(y_k)+\ell(y_k)
  \ge
  \liminf_{k\to\infty}
  \int_\O \frac{1}{(\det\nabla y_k(x))^s}\,\md x
  \ge
  \int_\O \frac{1}{(\det\nabla y(x))^s}\,\md x.
\end{align*}
Hence, inevitably, $\det\nabla y>0$ almost everywhere in $\O$
and $(\det\nabla y)^{-s}\in L^1(\O)$.
Since the trace operator is continuous, we obtain that $y\in\mathcal{A}$.

By Cramer's rule, the inverse of the deformation gradient satisfies
for almost all $x\in\O$ that
\begin{align} \label{inv-conv}
  (\nabla y_k(x))^{-1}
  =
  \frac{(\cof\nabla y_k(x))^\top}{\det\nabla y_k(x)}
  \longrightarrow
  \frac{(\cof\nabla y(x))^\top}{\det\nabla y(x)}
  =
  (\nabla y(x))^{-1}.
\end{align}   
Notice that, for almost all $x\in\O$
\begin{align*}
  \sup_{k\in\N} |\nabla y_k(x)|
  &=
  \sup_{k\in\N} \det\nabla y_k(x) \ |(\cof(\nabla y_k(x))^{-\top}| \nonumber \\
  &\le
  \sup_{k\in\N} \frac32 \det\nabla y_k(x) \ |(\nabla y_k(x))^{-1}|^2
  <\infty
\end{align*}
because of the pointwise convergence of $\{\det\nabla y_k\}$
and \eqref{inv-conv}.
Consequently, we have due to \eqref{inv-conv} for almost all $x\in\O$
and $k\to\infty$
\begin{align*}
  \nabla y_k(x)
  &=
  (\cof(\nabla y_k(x))^{-\top}\det\nabla y_k(x)
  \longrightarrow
  (\cof(\nabla y(x))^{-\top}\det\nabla y(x)
  =
  \nabla y(x),
\end{align*}
where we have used that the cofactor of some matrix is invertible
whenever the matrix itself is invertible too. 
As the Lebesgue measure on $\O$ is finite,
we get by the Egoroff theorem, c.f. \cite[Thm.~2.22]{fonseca-leoni},
\begin{align} \label{y-conv}
  \nabla y_k\to\nabla y \text{ in measure}.
\end{align}

Since $\hat W$ is nonnegative and continuous
and $\hat W(F,\cdot,\cdot)$ is convex,
we may use \cite[Cor.~7.9]{fonseca-leoni}
to conclude from \eqref{y-conv} and \eqref{grad-det-cof-weak} that
\begin{align*}
  &\int_\O
    \hat W(\nabla y(x),\nabla\cof\nabla y(x),\nabla\det\nabla y(x)) \,
  \md x\nonumber\\
  &\le
  \liminf_{k\to\infty}
  \int_\O
    \hat W(\nabla y_k(x),\nabla\cof\nabla y_k(x),\nabla\det\nabla y_k(x)) \,
  \md x\ .
\end{align*}
To pass to the limit in the functional $-\ell$,
we exploit its weak lower semicontinuity.
Therefore, the~whole functional $I$ is weakly lower semicontinuous along
$\{y_k\} \subset \mathcal{A}$
and hence $y\in\mathcal{A}$ is a~minimizer of $I$.
\end{proof}

\begin{remark}
  Note that the~pointwise convergence \eqref{det-conv} of the~determinant,
  necessary for~obtaining the~crucial convergence \eqref{y-conv},
  was not achieved by~compact embedding, as it was done for $\cof \nabla y$ in \eqref{cof-conv}.
  Hence the~coercivity in~$\nabla [\det \nabla y]$ is~of~minor importance
  and can be relaxed, provided the~function $\hat{W}$ from \eqref{full-I}
  does not~depend on~its~last argument, c.f. \cite[Prop. 5.1]{bbmkas}.
  On~the~other hand, although only $\nabla [\cof \nabla y]$
  is necessary for~regularizing the~whole problem,
  making the~functional in~\eqref{full-I}
  dependent also on~$\nabla [\det \nabla y]$
  may be interesting from~the~applications' point of~view.
\end{remark}

Let $\mathcal{L}^3$ denote the Lebesgue measure in $\R^3$.
If $p>3$ and  $y\in W^{1,p}(\O;\R^3)$ is such that
$\det\nabla y>0$ almost everywhere in $\O$,
then the so-called Ciarlet-Ne\v{c}as condition
\begin{align}\label{c-n}
  \int_\O\det\nabla y(x)\,\md x\le\mathcal{L}^3(y(\O))\ 
\end{align}
derived in \cite{ciarlet-necas2}
ensures almost-everywhere injectivity of deformations.
If
\begin{align} \label{h-k}
  \frac{|\nabla y|^3}{\det\nabla y} \in L^{\delta}(\O)
\end{align}
for some $\delta>2$ and \eqref{c-n} holds,
then we~even get invertibility everywhere in $\O$
due to \cite[Theorem~3.4]{hencl-koskela}.
Namely, this then implies that $y$ is an open map.
Hence, we get the following corollary of Proposition~\ref{prop-grad-poly}. 

\bigskip

\begin{corollary}
\label{cor-grad-poly}
Let $\O\subset\R^3$ be a bounded Lipschitz domain, 
and let  $\Gamma=\Gamma_0\cup\Gamma_1$ be a $\md A$-measurable partition of $\Gamma=\partial\O$ with the area of $\Gamma_0>0$. Let further $\ell:W^{1,p}(\O;\R^3)\to\R$ be a weakly upper semicontinuous functional  and $J$ as in \eqref{full-I} be gradient polyconvex on $\Omega$ such that $W$ satisfies \eqref{growth-graddet2}.  Finally, let  $p> 6$, $q\ge\frac{p}{p-1}$, $r>1$,  $s>2p/(p-6)$,
and assume that for some given measurable function $y_{\text{D}}: \Gamma_{\text{D}} \to \R^3$ the following set 
 \begin{align*}
\mathcal{A}:&=\{y\in W^{1,p}(\O;\R^3):\ \cof \nabla y\in W^{1,q}(\O;\R^{3\times 3}),\ \det \nabla y\in W^{1,r}(\O),\nonumber\\
& \qquad (\det\nabla y)^{-s}\in L^1(\O),\  \det\nabla y>0\mbox{ a.e. in $\Omega$},\ y=y_{\text{D}}\mbox{ on }\Gamma_{\text{D}},\, \eqref{c-n}\mbox{ holds}\}
\end{align*}
is nonempty.
 If $\inf_{\mathcal{A}} I<\infty$  for $I$ from \eqref{functional-sma} then
 the functional  $I$  has a minimizer on $\mathcal{A}$ which is injective everywhere in $\O$.
\end{corollary}

A~simple example of an energy density which satisfies the assumptions of Proposition~\ref{cor-grad-poly} and Corollary~\ref{cor-grad-poly} is 
\begin{align*}
\hat W(F,\Delta_1,\Delta_2)=\begin{cases}
W(F)+\varepsilon\big(|F|^p  +|\cof F|^q+(\det F)^r+ (\det F)^{-s}+|\Delta_1|^q + |\Delta_2|^r\big)
&\!\!\!\!\text{if }\det F>0,\\
+\infty&\text{ otherwise} \end{cases}
\end{align*}
for $W$ defined in \eqref{stored-sma}.

\begin{remark}[Gradient-polyconvex materials and smoothness of stress]\label{rem:transfer}
Gradient-polyconvex materials enable us to control
regularity of the first Piola-Kirchhoff stress tensor
by means of smoothness of the Cauchy stress. 
Assume that the Cauchy stress tensor $T^y:y(\O)\to\R^{3\times 3}$
is Lipschitz continuous, for instance.
If $\cof\nabla y:\Omega\to\R^{3\times 3}$ is Lipschitz continuous too,
then the first Piola-Kirchhoff stress tensor $S$ inherits
the Lipschitz continuity from $T^y$ because 
\begin{align*}
S(x):=T^y(x^y)\cof\nabla y(x)\ ,
\end{align*}
where $x^y:=y(x)$.
In a similar fashion, one can transfer H\"{o}lder continuity of $T^y$
to $S$ via H\"{o}lder continuity of $x\mapsto\cof\nabla y$.  
\end{remark}

In~literature, examples of stored energy density functions in nonlinear elasticity
are usually minimized on ${\rm SO}(3)$.
In the context of shape-memory alloys, the stored energy density is minimized 
on~${\rm SO}(3)F_i$, $F_i\ne F_j$, $i,j=0,\ldots,M$.
To~construct such energy densities explicitly, we can now proceed as follows.
Assume that $V:\R^{3\times 3}\to\R\cup\{+\infty\}$ is minimized
on~${\rm SO}(3)$ and that $V(F)=\varphi(F^\top F)=\varphi(C)$
for~some function $\varphi : \mathbb{R}^{3 \times 3}_{sym} \to \mathbb{R} \cup \{ +\infty \}$
and $C = F^\top F$ the~right-Cauchy-Green tensor.
It~is easy to~see that $\varphi$ is~minimized in~$\mathbb{I}$.
Considering the polar decomposition of~$F_i\in\R^{3\times 3}$ with $\det F_i>0$,
we can write $F_i=R_iU_i$ where $R_i$ is a rotation and $U_i$ is symmetric and positive definite matrix.
Note that $C_i = U_i^2$.
Bearing this in~mind, we define the~energy of~the~$i$-th variant via a~shift
\begin{align*}
  W_i(F) := V(FU_i^{-1}) = \varphi(U_i^{-1}CU_i^{-1})
\end{align*}
which is clearly minimized on~${\rm SO}(3)F_i$.
Notice also that if $V$ is polyconvex, so is $W_i$.


\section{Evolution}
If the loading changes in time
or if the boundary condition becomes time-dependent,
then the~specimen evolves as well. 
Evolution is typically connected with energy dissipation.
Experimental evidence shows
that considering a rate-independent dissipation mechanism
is a reasonable approximation
in a wide range of rates of external loads.
We hence need to define a suitable dissipation function.
Since we consider a rate-independent processes,
this dissipation will be positively one-homogeneous.
We associate the~dissipation to~the~magnitude
of~the~time derivative of~the~dissipative variable
$z \in \mathbb{R}^{M+1}$, where $M \in \N$,
i.e. to $|\dot{z}|_{M+1}$,
where $|\cdot|_{M+1}$ denotes a~norm on~$\mathbb{R}^{M+1}$
(in~our setting,
the~internal variable $z$ can be seen as~a~vector of~volume fractions
of~austenite and variants of~martensite).
Therefore, the specific dissipated energy
associated to a~change from state $z^1$ to $z^2$ is postulated as
\begin{align*}
  D(z^1,z^2):=|z^1-z^2|_{M+1}.
\end{align*}
Hence, for $z^i : \Omega \to \mathbb{R}^{M+1}$, $i = 1, 2$, the total dissipation reads
\begin{align*}
  \mathcal{D}(z^1,z^2):=\int_\O D(z^1(x),z^2(x))\ dx \ ,
\end{align*}
and the~total $\DD$-dissipation of a~time dependent curve $z : t \in [0,T] \mapsto z(t)$,
where $z(t) : \Omega \to \mathbb{R}^{M+1}$, is defined as
\begin{align*}
  {\rm Diss}_\DD(z,[s,t])
  :=
  \sup \Big \{ \sum_{j=1}^N \DD(z(t_{i-1}),z(t_i)) : N \in \N, s = t_0 \leq \ldots \leq t_N = t \Big \}
\end{align*}

Let $\mathcal{Z}$ denote the~set of~all admissible states of~internal variables $z : \Omega \to \mathbb{R}^{M+1}$
and $\mathcal{A}$ be the~set of~admissible deformations as before.
For a~given $(t,y,z)\in[0,T]\times\VV\times\ZZ$ we define the total energy of~the~system by
\begin{align*}
  \mathcal{E}(t,y,z) =
    \left\{
      \begin{array}{ll}
        \displaystyle
        J(y) - L(t,y) & \text{if } z = \lambda(\nabla y) \text{ a.e. in } \Omega, \\
        +\infty          & \text{otherwise,}
      \end{array}
    \right. ,
\end{align*}
where $L(t,\cdot)$ is a~functional on~deformations expressing time-dependent loading of~the~specimen
and $\lambda : \mathbb{R}^{3 \times 3} \to \mathbb{R}^{M+1}$
is~a~function relating the~deformation gradient with the~internal variable $z$.
For~example, we can define the~$j$th component of~$\lambda \in \R^{M+1}$ as
\begin{align*}
   \lambda^j (F)
   :=
   \frac1M \left(1- \frac{\dist(C,\NN(C_j))}{\sum_{i=0}^M \dist(C,\NN(C_i))}\right)
   \quad
   \forall C = F^T F\in \R^{3\times 3}, \quad j=0,\ldots, M\ ,
\end{align*}
where $\NN(C_i)$ are pairwise disjoint neighborhoods of $C_i$, $i=0,\ldots,M$.
\begin{remark}
The particular choice of $\lambda$ allows for some elastic behavior close to the wells $SO(3)F_i$, $i=0,\ldots, M$.
Note that $\lambda$ is continuous and frame-indifferent, and $\sum_{j=0}^M \lambda_j(F) = 1$ for all $F\in\R^{3\times 3}$.
\end{remark}

\section{Energetic solution}

Suppose, that we look for the time evolution of $t\mapsto y(t)\in \VV$
and $t\mapsto z(t)\in \ZZ := L^\infty(\Omega, \R^{M+1})$ during a process on~a~time interval $[0,T]$,
where $T>0$ is the time horizon.
We use the following notion of solution from \cite{FrancfortMielke-2006}, see also
\cite{mielke-theil-2, mielke-theil-levitas}.
For a~given energy $\mathcal{E}$, dissipation distance $\mathcal{D}$
and every admissible configuration living in 
\begin{align*}
 \QQ :=\{ (y,  z) \in  \VV\times \ZZ :  \lambda(\nabla y) = z \text{ a.e. in } \Omega \} 
\end{align*}
we ask the following conditions to be satisfied.
\begin{definition}[Energetic solution]
  We say that $(y,z) : [0,T] \to \QQ$ is an energetic solution to~$(\mathcal{Q},\EE,\DD)$
  if $t \mapsto \p_t \EE(y(t),z(t)) \in L^1(0,T)$
  and if for all $t \in [0,T]$ the stability condition
\begin{align}
  &\EE(t, y(t),z(t))\leq \EE(t,\tilde y,\tilde z)+\mathcal{D}(z(t),\tilde z) \qqquad %
  \text{$\forall (\tilde y,\tilde z)\in\QQ$}. \tag{S}
\end{align}
and the~energy balance
  \begin{align}  \tag{E}      %
    \begin{aligned}
    &\EE(t, y(t),z(t))+{\rm Diss}_\DD(z;[s,t]) =\EE(s,y(s),z(s)) +\displaystyle \int_0^t \partial_t \EE(s, y(s),z(s))\,\md s \hspace{-2ex}   
    \end{aligned} 
  \end{align}
  are satisfied.
\end{definition}

An important role is played by the~set of~so-called stable states,
defined for each $t\in[0,T]$ as
\begin{align*}
  \mathbb{S}(t)
  :=
  \{
    (y,z)\in\QQ: \,
    \EE(t,y,z) < +\infty \text{ and }
    \EE(t,y,z)\leq \EE(t,\tilde y,\tilde z)+\mathcal{D}(z,\tilde z)\,\forall (\tilde y,\tilde z)\in\QQ
  \}\ .
\end{align*}

\subsection{Existence of the energetic solution}

A standard way how to prove the existence of an energetic solution is to construct time-discrete minimization problems
and then to pass to the limit. Before we give the existence proof we need some auxiliary results. For given $N\in\N$ and
for $0\leq k\leq N$, we define the time increments $t_k:=kT/N$. Furthermore, we use the abbreviation $q:=(y,z)\in\QQ$.
We assume that there exists an~admissible deformation $y^0$ compatible with the~initial volume fraction
$z^0$, i.e. $q^0:=(y^0,z^0) \in \mathbb{S}(0)$.
For $k=1,\ldots, N$, we define a sequence of minimization problems
\begin{align}\label{incremental}
  \text{minimize } \II_k(y,z) := \EE(t_k, y,z)+\DD(z,z^{k-1})\ ,\ (y,z)\in \QQ\ .
\end{align}
We denote a minimizer of \eqref{incremental} for a given $k$ as $q^N_k :=(y^k,z^k)\in\QQ$ for $1 \leq k \leq N$.  The following lemma shows that
a minimizer always exists if the elastic energy is not identically infinite on $\QQ$.

\begin{lemma} \label{lem-dav} %
Let $\O\subset\R^3$ be a bounded Lipschitz domain, 
and let  $\Gamma=\Gamma_0\cup\Gamma_1$ be a $\md A$-measurable partition of $\Gamma=\partial\O$ with the area of $\Gamma_{\text{D}}>0$. 
Let $J$, of the~from \eqref{full-I}, be gradient polyconvex on~$\Omega$
and such that the~stored energy density $W$ satisfies \eqref{growth-graddet2}.
Moreover, let $L \in C^1([0,T]; W^{1,p}(\Omega;\R^3))$ be such that for some $C>0$ and $1\leq \alpha < p$
\begin{align*}
  L(t,y) \leq C \| y\|^\alpha_{W^{1,p}} \quad\forall t\in [0,T] \ 
\end{align*}
and $y\mapsto -L(t,y)$ be weakly lower semicontinuous on $W^{1,p}(\Omega;\R^3)$ for all $t\in [0,T]$.
Finally, let $p> 6$, $q\ge\frac{p}{p-1}$, $r>1$, $s>2p/(p-6)$.
 
If there is $(y,z)\in \QQ$ such that $\II_k(y,z)<\infty$
for $\II_k$ from \eqref{incremental},
then the functional $\II_k$ has a~minimizer $q^N_k=(y^k,z^k)\in \QQ$
such that $y_k$ is injective everywhere in $\O$.
Moreover, $q^N_k\in \mathbb{S}(t_k)$ for~all $1 \leq k \leq N$.  
\end{lemma}

\begin{proof}
Since the discretized problem \ref{incremental} has a~purely static character,
we can follow the proof of~Proposition~\ref{prop-grad-poly}.
Let $\{(y^k_j, z^k_j)\}_{j\in\N} \subset \QQ$ be a minimizing sequence.
As
\begin{align*}
 \nabla y^k_j \longrightarrow \nabla y^k
 \quad
 \mbox{ strongly in }
 L^{\tilde p}(\Omega, \R^{3\times 3}) \mbox{ as } j\to \infty
\end{align*}
for every $1 \leq \tilde p< p$
and $\lambda \in C(\R^{3\times 3}, \R^{M+1})$ is bounded,
we obtain that 
\begin{align*}
 z^k_j = 
 \lambda(\nabla y^k_j) \longrightarrow \lambda(\nabla y^k)
 \quad
 \mbox{ strongly in } L^{\tilde p}(\Omega, \R^{M+1}) 
\mbox{ as } j \to \infty \ .
\end{align*}
Since $\| z^k_j\|_{L^1(\Omega,\R^{M+1})}$ is uniformly bounded in $j$,
there is a subsequence such that $z^k_j \stackrel{*}{\longrightarrow} \mu^k$
in Radon measures on $\Omega$.
This shows that $z^k:= \mu^k = \lambda(\nabla y^k)$
and hence $q^N_k=(y^k,z^k) \in \mathcal{Q}$.
Since $\DD(\cdot, z^{k-1})$ is convex,
we obtain that $q^N_k$ is indeed a~minimizer of~$\mathcal{I}_k$.
Moreover $y_k$ is injective everywhere by the~reasoning
used for~proving Corollary \ref{cor-grad-poly}.
The~stability $q^N_k\in \mathbb{S}(t_k)$ follows by standard arguments,
see e.g. \cite{FrancfortMielke-2006}.
\end{proof}

Denoting by $B([0,T];\YY)$ the set of bounded maps
$t\mapsto y(t)\in\YY$ for all $t\in[0,T]$,
we have the~following result showing the existence of an energetic solution
to the problem $(\mathcal{Q},\EE,\DD)$.
\begin{theorem} \label{thm:main}
  Let $T>0$ and let the assumptions in Lemma~\ref{lem-dav} be satisfied.
  Moreover, let the initial condition be stable,
  i.e. $q^0 := (y^0,z^0) \in \mathbb{S}(0)$.
  Then there is an energetic solution to $(\mathcal{Q},\EE,\DD)$
  satisfying $q(0) = q^0$
  and such that $y\in B([0,T];\YY)$,
  $z\in {\rm BV}([0,T]; L^1(\O;\R^{M+1}))\cap L^\infty(0,T;\ZZ)$, and 
  for all $t\in[0,T]$ the identidy $\lambda(\nabla y(t,\cdot))=z(t,\cdot)$ holds a.e.\ in~$\O$.
  Moreover, for all $t \in [0,T]$ the~deformation $y(t)$
  is injective everywhere in~$\O$.
\end{theorem}
\begin{proof}
  Let $q^N_k:=(y^k,z^k)$ be the solution of \eqref{incremental}
  which exists by Lemma \ref{lem-dav} and let
  $q^N:[0,T]\to \QQ$ be given by
  \begin{eqnarray*}
    q^N(t):= %
    \begin{cases}
      q^N_k &\mbox{ if $t\in [t_{k},t_{k+1})$ if $k=0,\ldots, N-1$}\ ,\\
      q^N_N &\mbox{ if $t=T$.}
    \end{cases}
  \end{eqnarray*} 
  Following \cite{FrancfortMielke-2006},
  we get for some $C>0$ and for all $N\in\N$ the estimates
  \begin{subequations}
    \begin{gather} \label{unif-est-z}
      \|z^N\|_{BV(0,T; L^1(\O;\R^{M+1}))}\leq C, \qquad
      \|z^N\|_{L^\infty(0,T; BV(\O;\R^{M+1}))}\leq C, \\ \label{unif-est-y}
      \|y^N\|_{L^\infty(0,T;W^{1,p}(\O;\R^3))}\leq C, 
    \end{gather}
  \end{subequations}
  as well as the following two-sided energy inequality
  \begin{align} \label{2-sided}
    \int_{t_{k-1}}^{t_k}\partial_t\mathcal{E}(\theta,q_{k}^N)\,{\rm d}\theta
    &\leq
    \mathcal{E}(t_k,q^N_k)+\mathcal{D}(z^k,z^{k-1})
    -\mathcal{E}(t_{k-1},q^N_{k-1}) \NT \\
    &\leq
    \int_{t_{k-1}}^{t_k}
      \partial_t\mathcal{E}(\theta,q_{k-1}^N) \,
    {\rm d}\theta\ .
  \end{align}
  The second inequality in \eqref{2-sided} follows since
  $q_{k}^N$ is a minimizer of \eqref{incremental} and by
  comparison of its energy with $q:=q_{k-1}^N$.
  The lower estimate is implied by the stability of
  $q^N_{k-1}\in\mathbb{S}(t_{k-1})$, see Lemma~\ref{lem-dav},
  when compared with $\tilde q:=q^N_{k}$.
  Having this inequality, the a-priori estimates
  and a generalized Helly's selection principle
  \cite[Cor.~2.8]{mielke-theil-levitas}, we get that there is
  indeed an~energetic solution obtained as a~limit for $N\to \infty$. 

  Let us comment more on~the~two main properties of~the~minimizer,
  namely that it is orientation preserving and injective everywhere in~$\O$.
  The condition $\det\nabla y>0$ a.e.\ in $\Omega$
  follows from the fact that if
  $t_j\to t$, $(y_{(j)}, z_{(j)}) \in \mathbb{S}(t_j)$
  and $(y_{(j)}, z_{(j)}) \wto (y,z)$
  in $W^{1,p}(\Omega;\R^3) \times BV(\O;\R^{M+1})$,
  then $(y,z) \in \mathbb{S}(t)$.
  Indeed, we have $z_{(j)}\to z$ in~$L^1(\O;\R^{M+1})$ in our setting
  and hence for all $(\tilde y,\tilde z)\in\QQ$, we get
  \begin{align*}
    \EE(t,y,z) %
    &\leq
    \liminf_{j\to \infty}\EE(t_j,y_{(j)}, z_{(j)}) %
    \leq
    \liminf_{j\to \infty} (\EE(t_j,\tilde y, \tilde z)
      + \DD(z_{(j)},\tilde z))\\
    &=
    \EE(t,\tilde y, \tilde z)+\DD(z,\tilde z)\ .
  \end{align*}
  In particular, as $\EE(t_j,\tilde y, \tilde z)$ is finite for some
  $(\tilde y,\tilde z)\in\QQ$,
  we get $\EE(t,y,z)<+\infty$ and thus $\det\nabla y>0$ a.e. in $\Omega$
  in view of \eqref{growth-graddet2}.

  In~proving injectivity, we profit again from the~fact
  that quasistatic evolution of energetic solutions
  is very close to~a~purely static problem.
  In~view of~\eqref{unif-est-y}, we obtain for each $t \in [0,T]$
  all necessary convergences
  that were used in the proof of~Corollary \ref{cor-grad-poly}
  to pass to~the~limit in~the~conditions \eqref{c-n} and \eqref{h-k}.
\end{proof}

\bigskip

{\bf Acknowledgment:}
This research was partly supported by the GA\v{C}R grants 17-04301S and 18-03834S, and by the DAAD-AV\v{C}R grant DAAD 16-14 and PPP 57212737 with funds from
BMBF.
PP moreover gratefully acknowledges the financial support
by GAUK project No.~670218,
by Charles University Research program No.~UNCE/SCI/023, and by GA\v{C}R-FWF project 16-34894L.

\end{document}